\documentclass[a4paper,12pt]{amsart}
\usepackage{graphicx}
\usepackage[psamsfonts]{amssymb}
\usepackage{amscd,pifont}
\usepackage[usenames]{color}
\usepackage{amssymb,amsmath,amsthm}
\usepackage[colorlinks,linkcolor=red,anchorcolor=blue,citecolor=blue]{hyperref}
\makeatletter
\@namedef{subjclassname@2020}{\textup{2020} Mathematics Subject Classification}
\makeatother

\newtheorem{theorem}{Theorem}[section]
\newtheorem{lemma}[theorem]{Lemma}
\newtheorem{proposition}[theorem]{Proposition}
\newtheorem{corollary}[theorem]{Corollary}
\newtheorem{definition}[theorem]{Definition}

\newcommand{\abs}[1]{\left| #1 \right|}

\def\N{\mathbb{N}}
\def\Z{\mathbb{Z}}

\def\Zp{\mathbb{Z}_{p}}
\def\Qp{\mathbb{Q}_{p}}

\begin{document}

\title{A characterization  of compact open spectral sets in $\Qp^d$}

\author{Mamateli Kadir}
\address{School of Mathematics and Statistics, Kashi University, Kashi, Xinjiang, 844000  China}

\email{mamatili880@163.com}


\begin{abstract}
Let $\Omega \subset \Qp^d$ be a compact open set. Such a set, without lose of generality,  admits a representation
\(\Omega = \bigsqcup_{c \in C} (c + p^n\Zp^d)\), where $C \subset (\Z/p^n\Z)^d$ and $n \in\N$.
We prove that $\Omega$ tiles $\Qp^d$ by translation if and only if $C$ tiles $(\Z/p^n\Z)^d$ by translation.
Moreover, $\Omega$ is a spectral set in $\Qp^d$ if and only if $C$ is a spectral set in $(\Z/p^n\Z)^d$.
\end{abstract}

\subjclass[2020]{Primary 43A25;  Secondary 43A70, 26E30.}

\keywords{$p$-adic space, compact open set, spectral set, translational tile.}

\maketitle

\section{Introduction}

The celebrated \textbf{Fuglede's conjecture}, also known as the \textbf{spectral set  conjecture},
proposed by Bent Fuglede in 1974~\cite{Fuglede1974}. It stands as one
of the most intriguing problems at the intersection of harmonic analysis, spectral theory, and geometric measure theory.
Originally formulated for Euclidean spaces, the conjecture states that a Lebesgue measurable set \(\Omega \subset \mathbb{R}^d\) of
positive finite measure is a \textbf{spectral set}, that is, the Hilbert space \(L^2(\Omega)\)
admits an orthogonal basis of exponential functions
\(\{e^{2\pi i \lambda \cdot x} : \lambda \in \Lambda\}\) for some discrete set \(\Lambda \subset \mathbb{R}^d\),
if and only if it \textbf{tiles} \(\mathbb{R}^d\) by translation. The latter means there exists a set \(T \subset \mathbb{R}^d\)
such that \(\{\Omega + t\}_{t \in T}\) partitions \(\mathbb{R}^d\) up to null sets. This elegant statement bridges the
seemingly disparate concepts of spectral theory and translational tilings, suggesting a profound duality between
analytic and combinatorial structures.

The conjecture was originated from Fuglede's study of extensions of partial differential operators
acting on  $C_c^\infty(\Omega)$ to commuting self-adjoint operators on $L^2(\Omega)$, where spectral
sets appear naturally as domains supporting such operators.

Early positive results included Fuglede's own proof under
the assumption that either the spectrum \(\Lambda\) or the tiling set \(T\) is a lattice~\cite{Fuglede1974}. Subsequent
research revealed both confirmations and counterexamples, painting a nuanced picture of the conjecture's validity.

In the positive direction, it was shown that in \(\mathbb{R}^2\), convex planar domains are spectral exactly when they
are parallelograms or centrally symmetric hexagons~\cite{Iosevich-Katz-Tao2008}. This confirmed the conjecture for convex
domains in the plane, recently   it was  proved for convex domains in $\mathbb{R}^d$ for all $d\geq1$ \cite{LM21}.
Additionally, for specific classes such as the unit cube \([0,1]^d\), it was established that
spectrality and tiling are equivalent~\cite{Iosevich-Pedersen1998,Lagarias-Reed-Wang2000}.
For unions of two intervals in \(\mathbb{R}\),  the conjecture also holds~\cite{Laba2001}.

On the negative side, several classes of sets were shown not to be spectral, such as triangles~\cite{Fuglede1974}, any
non-symmetric convex domains~\cite{Kolountzakis2000}, balls in \(\mathbb{R}^d\) for \(d \ge 2\)~\cite{Fuglede2001},
and any convex domains with smooth boundary~\cite{Iosevich-Katz-Tao}. More dramatically, Tao~\cite{Tao2004} disproved
the conjecture in high dimensions by constructing a spectral set in \(\mathbb{R}^5\) that does not tile.
Subsequent work produced counterexamples for both directions in dimensions \(d \ge 3\)~\cite{FMM2006,KM2006,KM2006a,Matolcsi2005}.
As of today, the conjecture remains open in dimensions \(d = 1\) and \(2\) for both implications,
though significant partial results exist.

A natural generalization of Fuglede's conjecture to locally compact abelian (LCA) groups has attracted
considerable attention. Let \(G\) be  a second countable LCA group with Haar measure \(\mathfrak{m}\),
and  $\widehat{G}$ be the dual group of $G$ consisting of the continuous characters of $G$. A Borel set
\(\Omega \subset G\) with \(0<\mathfrak{m}(\Omega)<\infty\) is called spectral if there exists
a set $\Lambda\subset\widehat{G}$ forming an orthogonal  basis for the space $L^2(\Omega)$.
A set \(\Omega \subset G\) is said to tile \(G\) by translation if there exists a set \(T \subset G\)
such that \(\{ \Omega + t : t \in T \}\) partitions \(G\) up to measure zero.

This generalized formulation reveals connections with diverse areas including abstract harmonic analysis,
combinatorics, and operator theory. It also allows for testing the conjecture in settings with different
topological and algebraic structures, potentially revealing which aspects of Euclidean space are essential
for the conjecture to hold.

The field of \(p\)-adic numbers \(\mathbb{Q}_p\) and its vector spaces \(\mathbb{Q}_p^d\) offer a particularly
compelling setting for investigating the Fuglede's conjecture. As a non-Archimedean local field, \(\mathbb{Q}_p\)
possesses a totally disconnected topology that imposes strong regularity conditions on measurable sets.
Unlike \(\mathbb{R}^d\), where spectral sets and tiles can have highly irregular boundaries, in \(\mathbb{Q}_p\)
any set of positive Haar measure that is either spectral or a tile must be \textbf{compact open} and
\textbf{$p$-homogeneous}~\cite{FFS}. 

The harmonic analysis in \(\mathbb{Q}_p^d\) is both rich and tractable. The dual group \(\widehat{\mathbb{Q}}_p^d\) is isomorphic to \(\mathbb{Q}_p^d\) itself, and Fourier transforms exhibit a clean duality between local constancy and compact support.
Furthermore, the absence of traditional lattices is compensated by the notion of \textbf{quasi-lattices},  a discrete set that play an
analogous role in the \(p\)-adic setting.

Recent work has established that the Fuglede conjecture holds in one dimensional space \(\mathbb{Q}_p\).
However, this conjecture is false in \(\mathbb{Q}_p^d\) for \(d\geq3\), and remains  open in \(\mathbb{Q}_p^2\) \cite{FFLS}.
The proof in \(\mathbb{Q}_p\) relies crucially on the compact open structure of such sets. Our present work deepens this
understanding by establishing an explicit reduction to finite combinatorial problems.

Let \(\Omega \subset \mathbb{Q}_p^d\) be a compact open set. By standard properties of the \(p\)-adic topology,
after translation and dilation we may assume
\[
\Omega = \bigsqcup_{c \in C} \left( c + p^{n} \mathbb{Z}_{p}^{d} \right),
\]
where \(n \in \mathbb{N}\) and \(C\) is a finite subset of the quotient group \((\mathbb{Z}/p^n\mathbb{Z})^d\),
identified with digit vectors in \(\{0,1,\dots,p^n-1\}^d\). This representation expresses \(\Omega\) as a
disjoint union of cosets of the compact open subgroup \(p^n\mathbb{Z}_p^d\), with \(C\) serving as the ``digit set"
determining which cosets are included.

Our main results establishes a precise equivalence between the tiling and spectral properties of \(\Omega\)
in the continuous \(p\)-adic setting and those of its digit set \(C\) in the finite abelian group
\(G = (\mathbb{Z}/p^n\mathbb{Z})^d\).

\begin{theorem}\label{main}
Let \(\Omega\) and \(C\) be as above.
\begin{enumerate}
    \item \(\Omega\) tiles \(\mathbb{Q}_p^d\) by translation if and only if \(C\) tiles \(G\) by translation.
    \item \(\Omega\) is a spectral set in \(\mathbb{Q}_p^d\) if and only if \(C\) is a spectral set in \(G\).
\end{enumerate}
\end{theorem}

This reduction theorem has several important implications.

\textbf{Combinatorial interpretation.} It translates the continuous Fuglede's problem in \(\mathbb{Q}_p^d\)
into a purely combinatorial problem on finite abelian groups, where tools from finite Fourier analysis and
combinatorics become available.

\textbf{Counterexamples in \(\mathbb{Q}_p^d\).} Since there are some counterexamples of Fuglede's conjecture in
 finite abelian groups, Theorem~\ref{main} provides a powerful tool to construct  counterexamples
 in \(\mathbb{Q}_p^d\).

\textbf{Classification of spectral or tiling sets.} The theorem provides a complete characterization of compact open spectral sets and tiles in \(\mathbb{Q}_p^d\) in terms of their digit sets, enabling explicit constructions and classifications.

\textbf{Structural insight.} The equivalence reveals that the essential features determining whether a compact open set in \(\mathbb{Q}_p^d\) is spectral or tiles are captured entirely by the arithmetic-combinatorial properties of its digit set in the finite quotient group.


The proof of Theorem~\ref{main} employs several key techniques from \(p\)-adic analysis.

A crucial technical lemma (Lemma \ref{lem:spectral-equivalence}) establishes that \((C, \Lambda)\) is a spectral pair
in the finite group \((\mathbb{Z}/p^n\mathbb{Z})^d\) if and only if \((\delta_C, \frac{1}{p^n}\Lambda)\) is a spectral pair in \(\mathbb{Q}_p^d\), where \(\delta_C\) is the uniform probability measure on \(C\). This connection between finite group
spectrality and \(p\)-adic spectrality for atomic measures forms the bridge between the discrete and continuous settings.

This work contributes to a growing body of research on spectral sets in totally disconnected LCA groups. The reduction approach developed here may extend to other local fields and more general totally disconnected groups. Additionally, the explicit connection between \(p\)-adic spectral sets and finite combinatorial structures opens avenues for computational investigation and explicit classification of examples.

Beyond the Fuglede's conjecture itself, understanding spectral sets in non-Archimedean settings has potential applications in \(p\)-adic mathematical physics, wavelet theory on local fields, and number theory. The structural results obtained here may also inform the study of related problems such as the existence of Riesz bases of exponentials and the characterization of bounded spectral sets in general LCA groups.

By clarifying the precise relationship between continuous \(p\)-adic spectral and tiling problems and their finite combinatorial counterparts, this work provides both a concrete tool for classification and a conceptual framework for understanding spectrality in totally disconnected settings.

\textbf{Organization of the paper}.
Section~\ref{prelim}  provides necessary background on \(p\)-adic analysis, including the structure
of \(\mathbb{Q}_p\), quasi-lattices, Fourier analysis in \(\mathbb{Q}_p^d\), and the fundamental
criteria for spectral sets and tiles. Section~\ref{mainresult} contains the proofs of our main results. We first
establish the tiling equivalence, then prove the spectral equivalence,
and finally combine these to prove Theorem~\ref{main}.

\section{Preliminaries}\label{prelim}

This section provides a comprehensive foundation in the mathematical concepts and tools essential for
understanding and proving our main results. We begin with the structure of the \(p\)-adic number field,
proceed through harmonic analysis in \(\mathbb{Q}_p^d\), introduce the key notions of quasi-lattices
and uniform discreteness, establish criteria for spectral sets and tilings, and finally recall relevant
results from finite harmonic analysis that bridge the discrete and continuous settings.

\subsection{The $p$-adic number field $\Qp$}


Let \(p\) be a fixed prime number. For any nonzero rational number \(r \in \mathbb{Q}^\times\),
we can write uniquely
\[
r = p^{v_p(r)} \cdot \frac{a}{b},
\]
where \(a, b \in \mathbb{Z}\) are coprime to \(p\), and \(v_p(r) \in \mathbb{Z}\) is the \textbf{\(p\)-adic valuation}.
We extend \(v_p\) to all of \(\mathbb{Q}\) by setting \(v_p(0) = +\infty\).
The \textbf{\(p\)-adic absolute value} is then defined as
\[
|r|_p = p^{-v_p(r)} \quad (r \neq 0), \qquad |0|_p = 0.
\]
This absolute value satisfies the following properties for all \(r, s \in \mathbb{Q}\):

\begin{enumerate}
    \item \textbf{Positivity:} \(|r|_p \ge 0\), with equality iff \(r = 0\).
    \item \textbf{Multiplicativity:} \(|rs|_p = |r|_p |s|_p\).
    \item \textbf{Strong triangle inequality:} \(|r + s|_p \le \max\{|r|_p, |s|_p\}\).
\end{enumerate}

The third property makes \(|\cdot|_p\) a \textbf{non-Archimedean} absolute value. It implies that \(\mathbb{Q}\) becomes an ultrametric space. Where every triangle is isosceles, if \(|r|_p \neq |s|_p\), then \(|r + s|_p = \max\{|r|_p, |s|_p\}\).

The field of \textbf{\(p\)-adic numbers} \(\mathbb{Q}_p\) is defined as the completion of \(\mathbb{Q}\) with respect to \(|\cdot|_p\). Elements of \(\mathbb{Q}_p\) can be represented uniquely as formal Laurent series in \(p\):
\[
x = \sum_{n = v_p(x)}^{\infty} a_n p^n, \quad a_n \in \{0, 1, \dots, p-1\}, \quad a_{v_p(x)} \neq 0.
\]
The \textbf{fractional part} of \(x\) is defined as
\[
\{x\} = \sum_{n = v_p(x)}^{-1} a_n p^n \quad (\text{if } v_p(x) < 0), \qquad \{x\} = 0 \text{ if } v_p(x) \ge 0.
\]

The set \[\mathbb{Z}_p = \{x \in \mathbb{Q}_p : |x|_p \le 1\} = \{x \in \mathbb{Q}_p : v_p(x) \ge 0\}\] is the \textbf{ring of \(p\)-adic integers}, which is compact and open in \(\mathbb{Q}_p\). The maximal ideal of \(\mathbb{Z}_p\) is \[p\mathbb{Z}_p = \{x : |x|_p < 1\},\] and the residue field is \[\mathbb{Z}_p/p\mathbb{Z}_p \cong \mathbb{F}_p.\]


The topology of \(\mathbb{Q}_p\) is totally disconnected, the connected components are singletons.
A base for the topology is given by the balls
\[
B(a, p^{-n}) = \{x \in \mathbb{Q}_p : |x - a|_p \le p^{-n}\} = a + p^n \mathbb{Z}_p, \quad a \in \mathbb{Q}_p, \; n \in \mathbb{Z}.
\]
These balls are both open and compact (clopen). The Haar measure \(\mathfrak{m}\) in \(\mathbb{Q}_p\),
as an additive locally compact group,  is normalized so that \(\mathfrak{m}(\mathbb{Z}_p) = 1\).
Consequently, for any ball \(B(a, p^{-n})\), we have \(\mathfrak{m}(B(a, p^{-n})) = p^{-n}\).

%
%

For the \(d\)-dimensional vector space \(\mathbb{Q}_p^d\), we take the product topology and the product Haar measure, also denoted by \(\mathfrak{m}\) or \(dx\). We equip \(\mathbb{Q}_p^d\) with the maximum norm
\[
|x|_p = \max_{1 \le j \le d} |x_j|_p, \quad x = (x_1, \dots, x_d) \in \mathbb{Q}_p^d,
\]
and the scalar product
\[
x \cdot y = \sum_{j=1}^d x_j y_j \in \mathbb{Q}_p.
\]

\subsection{Characters and Fourier analysis in \(\mathbb{Q}_p^d\)}\label{subsec:fourier-analysis}


The standard additive character in \(\mathbb{Q}_p\) is defined by
\[
\chi(x) = e^{2\pi i \{x\}}, \quad x \in \mathbb{Q}_p.
\]
For each \(\xi \in \mathbb{Q}_p\), the character \(\chi_\xi\) is defined by \(\chi_\xi(x) = \chi(\xi x)\).
The mapping \(\xi \mapsto \chi_\xi\) gives an isomorphism between \(\mathbb{Q}_p\) and its dual group \(\widehat{\mathbb{Q}}_p\).
Thus we identify \(\widehat{\mathbb{Q}}_p\) with \(\mathbb{Q}_p\).

In higher dimensions, for \(x, y \in \mathbb{Q}_p^d\), we set
\[
\chi_y(x) = \chi(x \cdot y) = e^{2\pi i \{x \cdot y\}}.
\]
Then \(\widehat{\mathbb{Q}}_p^d = \{\chi_y : y \in \mathbb{Q}_p^d\} \cong \mathbb{Q}_p^d\).

Two key properties of these characters are frequently used.

\begin{enumerate}
    \item \textbf{Restriction to balls:} If \(x \in \frac{k}{p^n} + \mathbb{Z}_p\) with \(k, n \in \mathbb{Z}\), then
          \[
          \chi(x) = e^{2\pi i k / p^n}.
          \]
    \item \textbf{Orthogonality on subgroups:} For any \(n \ge 1\),
          \[
          \int_{p^{-n}\mathbb{Z}_p} \chi(x) \, dx = 0.
          \]
\end{enumerate}
The interested readers are referred to \cite{VVZ1994} for further information about characters of \(\mathbb{Q}_p^d\).


For \(f \in L^1(\mathbb{Q}_p^d)\), the Fourier transform is defined as
\[
\widehat{f}(\xi) = \int_{\mathbb{Q}_p^d} f(x) \overline{\chi_\xi(x)} \, dx = \int_{\mathbb{Q}_p^d} f(x)
\chi(-\xi \cdot x) \, dx, \quad \xi \in \mathbb{Q}_p^d.
\]
The inversion formula holds for sufficiently nice functions (e.g., \(f \in L^1 \cap L^2\)),
\[
f(x) = \int_{\mathbb{Q}_p^d} \widehat{f}(\xi) \chi_\xi(x) \, d\xi.
\]


A function \(f : \mathbb{Q}_p^d \to \mathbb{C}\) is called \textbf{locally constant}, if for every \(x\)
there exists an open neighborhood of \(x\) on which \(f\) is constant. If the radius of constancy can be chosen
uniformly independent of \(x\), then \(f\) is \textbf{uniformly locally constant}.

\begin{definition}\label{def:locally-constant}
A function \(f : \mathbb{Q}_p^d \to \mathbb{C}\) is uniformly locally constant if there exists \(n \in \mathbb{Z}\) such that
\[
f(x + y) = f(x) \quad \forall x \in \mathbb{Q}_p^d, \; \forall y \in B(0, p^n).
\]
Equivalently, \(f\) is constant on every coset of the compact open subgroup \(p^n\mathbb{Z}_p^d\).
\end{definition}

Uniformly locally constant functions are continuous, in fact, they are the analogue of ``smooth" functions
in the \(p\)-adic setting. The following proposition describes the duality between local constancy and
compact support under the Fourier transform.

\begin{proposition}[\cite{FFS}]\label{prop:duality}
Let \(f \in L^1(\mathbb{Q}_p^d)\).

\begin{enumerate}
    \item If \(f\) is uniformly locally constant, then \(\widehat{f}\) has compact support.
    \item If \(f\) has compact support, then \(\widehat{f}\) is uniformly locally constant.
\end{enumerate}
\end{proposition}

This proposition is fundamental because it implies that the Fourier transform of a compactly supported function is not only continuous but in fact ``smooth" in the \(p\)-adic sense. Consequently, the Fourier transform of the indicator function of a compact open set is uniformly locally constant and, by the inversion formula, is determined by its values on a discrete set.

\subsection{Quasi-lattices and uniformly discrete sets}\label{subsec:quasi-lattices}


Unlike \(\mathbb{R}^d\), the group \(\mathbb{Q}_p^d\) contains no discrete cocompact subgroups (lattices) because it is totally disconnected. Instead, we use the concept of \textbf{quasi-lattices}.

Let \(\mathbb{L}^d \subset \mathbb{Q}_p^d\) be a complete set of coset representatives of \(\mathbb{Q}_p^d / \mathbb{Z}_p^d\), so that
\[
\mathbb{Q}_p^d = \bigsqcup_{\gamma \in \mathbb{L}^d} (\gamma + \mathbb{Z}_p^d)=\mathbb{L}^d\oplus\mathbb{Z}_p^d.
\]
This implies that \(\mathbb{L}^d\) is a tiling complement of \(\mathbb{Z}_p^d\) in \(\mathbb{Q}_p^d\).
For each integer $n$, let
\[
\mathbb{L}_n^d=p^{-n}\mathbb{L}^d.
\]
Notice that
\[
\mathbb{Q}_p^d =p^{-n}\mathbb{Z}_p^d\oplus p^{-n}\mathbb{L}^d=B(0, p^n)\oplus \mathbb{L}_n^d.
\]
So $\mathbb{L}_n^d$ is a tiling complement of $B(0, p^n)$.
We call \(\mathbb{L}^d\) a \textbf{standard quasi-lattice}. More generally, if \(M \in \operatorname{GL}_d(\mathbb{Q}_p)\) is
a non-singular \(d\times d\) matrix,
then the set \(M\mathbb{L}^d\) is called a quasi-lattice. Quasi-lattices play the role of lattices in
the \(p\)-adic setting for harmonic analysis.

\begin{lemma}[\cite{Taibleson1975}]\label{lem:standard-basis}
The set of characters \(\{\chi_\gamma\}_{\gamma \in \mathbb{L}^d}\) forms an orthonormal basis for \(L^2(\mathbb{Z}_p^d)\).
\end{lemma}
Consequently, for any \(a \in \mathbb{Q}_p^d\), the system \(\{\chi_\gamma\}_{\gamma \in \mathbb{L}^d}\) is an orthonormal basis for \(L^2(a + \mathbb{Z}_p^d)\), so \((a + \mathbb{Z}_p^d, \mathbb{L}^d)\) is a spectral pair.


A set \(\Lambda \subset \mathbb{Q}_p^d\) is called \textbf{uniformly discrete} if there exists \(\delta \in \mathbb{Z}\)
such that for any two distinct points \(\lambda, \lambda' \in \Lambda\),
\[
|\lambda - \lambda'|_p \ge p^\delta.
\]
The largest such \(\delta\) is called the \textbf{separation constant}, and denoted by \(\delta(\Lambda)\).

Every quasi-lattice is uniformly discrete. For the standard quasi-lattice \(\mathbb{L}^d\), we have \(\delta(\mathbb{L}^d) = p\)~\cite{Fan}.

\subsection{A spectral set  criterion}\label{subsec:spectral-criteria}


Let \(\Omega \subset \mathbb{Q}_p^d\) be a Borel set with \(0 < \mathfrak{m}(\Omega) < \infty\). A countable set \(\Lambda \subset \mathbb{Q}_p^d\) is called a \textbf{spectrum} for \(\Omega\) if the system of exponential functions
\[
E(\Lambda) = \big\{ \chi_\lambda(\cdot)|_\Omega : \lambda \in \Lambda \big\}
\]
forms an orthogonal basis for \(L^2(\Omega)\). In this case, we say that \(\Omega\) is a \textbf{spectral set} and that \((\Omega, \Lambda)\) is a \textbf{spectral pair}.


Given a discrete set \(\Lambda\), the orthogonality of \(E(\Lambda)\) in \(L^2(\Omega)\) is equivalent to a condition on the difference set \(\Lambda - \Lambda\) and the Fourier transform of the indicator function \(1_\Omega\).

\begin{lemma}\label{lem:orthogonality}
Let \(\Omega\) be as above and \(\Lambda \subset \mathbb{Q}_p^d\) be countable. Then \(E(\Lambda)\) is an orthogonal system in \(L^2(\Omega)\) if and only if
\[
\Lambda - \Lambda \subset \mathcal{Z}(\widehat{1_\Omega}) \cup \{0\},
\]
where \(\mathcal{Z}(g) = \{\xi : g(\xi) = 0\}\).
\end{lemma}

\begin{proof}
For \(\lambda \neq \lambda'\), the inner product \(\langle \chi_\lambda, \chi_{\lambda'} \rangle_{L^2(\Omega)}\) equals \(\widehat{1_\Omega}(\lambda - \lambda')\). Orthogonality requires this to be zero.
\end{proof}


Completeness of an orthogonal system is more subtle. The following criterion, analogous to the characterizations in Euclidean spaces, is essential.

\begin{lemma}[\cite{Fan}]\label{lem:spectral-criterion}
Let \(\Omega \subset \mathbb{Q}_p^d\) be a Borel set with \(0 < \mathfrak{m}(\Omega) < \infty\) and
let \(\Lambda \subset \mathbb{Q}_p^d\) be countable. Then \((\Omega, \Lambda)\) is a spectral pair
if and only if
\begin{equation}\label{Spectralcriterion}
\sum_{\lambda \in \Lambda} |\widehat{\mathbf{1}_\Omega}(\xi-\lambda)|^2 = \mathfrak{m}(\Omega)^2 \quad \forall \xi \in \mathbb{Q}_p^d.
\end{equation}
\end{lemma}
The identity (\ref{Spectralcriterion}) is often referred  as a \textbf{spectral set criterion}. It implies both orthogonality (by evaluating at \(\xi = 0\)) and completeness (by density arguments).

Uniform discreteness is a necessary condition for a set to be a spectrum.

\begin{lemma}\label{lem:uniform-discrete}
Let \(\Omega \subset \mathbb{Q}_p^d\) be a Borel set with \(0 < \mathfrak{m}(\Omega) < \infty\).
If \((\Omega, \Lambda)\) is a spectral pair, then \(\Lambda\) is uniformly discrete.
\end{lemma}

\begin{proof}
By definition of a spectral pair, the set \(E(\Lambda)\)
is an orthonormal basis of \(L^2(\Omega)\).
For distinct \(\lambda, \lambda' \in \Lambda\), orthonormality gives
\[
\int_{\Qp^d} 1_{\Omega}(x) \overline{\chi(\lambda \cdot x)} \chi(\lambda' \cdot x) dx = 0.
\]
Rewriting the integrand using the character property \[\overline{\chi(a)} \chi(b) = \chi(b - a),\] this simplifies to
\[
\widehat{1_{\Omega}}(\lambda - \lambda') = 0.
\]

Since \(\widehat{1_{\Omega}}(0) = \mathfrak{m}(\Omega) > 0\) and \(\widehat{1_{\Omega}}\) is continuous, there exists an integer \(n_0 \in \mathbb{Z}\) such that \(\widehat{1_{\Omega}}(\xi)\neq 0\) for all \(\xi \in B(0,\ p^{n_0})\).

For distinct \(\lambda, \lambda' \in \Lambda\), \(\widehat{1_{\Omega}}(\lambda - \lambda') = 0\) implies \(\lambda - \lambda' \notin B(0,\ p^{n_0})\). Thus
\[
\left\vert {\lambda - \lambda'}\right\vert_p > p^{n_0} \implies \inf_{\lambda \neq \lambda' \in \Lambda}
\left\vert{\lambda - \lambda'}\right\vert_p  \geq p^{n_0} > 0.
\]
Hence, \(\Lambda\) is uniformly discrete.
\end{proof}


More generally, one can consider spectral pairs for probability measures. A Borel probability measure \(\mu\) in \(\mathbb{Q}_p^d\) is called \textbf{spectral} if there exists a countable set \(\Gamma \subset \mathbb{Q}_p^d\) such that \(E(\Gamma)\) is an orthonormal basis for \(L^2(\mu)\). The criterion (\ref{Spectralcriterion}) adapts naturally, \(\mu\) is spectral with spectrum \(\Gamma\) if and only if
\[
\sum_{\gamma \in \Gamma} |\widehat{\mu}(\gamma - \xi)|^2 = 1 \quad \forall \xi \in \mathbb{Q}_p^d.
\]

\subsection{Tiling and density of tiling sets}\label{subsec:tiling}


 A set \(\Omega \subset \Qp^d\) of positive finite measure is said to \textbf{tile} \(\Qp^d\)
 by translation if there exists a countable set \(T \subset \Qp^d\) such that
\begin{equation}\label{tilings}
\sum_{t \in T} 1_\Omega(x-t) = 1 \quad \text{for } \mathfrak{m}\text{-a.e. } x \in \mathbb{Q}_p^d.
\end{equation}
Equivalently, the family \(\{\Omega + t : t \in T\}\) partitions \(\mathbb{Q}_p^d\) up to null sets. We then write \(\Omega \oplus T = \mathbb{Q}_p^d\) and call \(T\) a \textbf{tiling set} for \(\Omega\).

Equation (\ref{tilings}) can be written as a convolution equation
\[
\mu_T * \mathbf{1}_\Omega = 1, \quad \text{where } \mu_T = \sum_{t \in T} \delta_t.
\]

Actually, a tiling set is also uniformly discrete.

\begin{lemma}\label{lem:uniform-discrete1}
Let \(\Omega \subset \mathbb{Q}_p^d\) be a Borel set with \(0 < \mathfrak{m}(\Omega) < \infty\).
If \((\Omega, T)\) is a tiling pair, then \(T\) is uniformly discrete.
\end{lemma}

\begin{proof}
By definition of a tiling pair, we have
\[\sum_{t \in T} 1_{\Omega}(x - t) = 1\ \text{for} \ \mathfrak{m}\text{-a.e. } x \in \Qp^d.\]
For distinct \(t, t' \in T\), the translates \(\Omega + t\) and \(\Omega + t'\)
are disjoint modulo \(\mathfrak{m}\), so
\[
\mathfrak{m}((\Omega + t) \cap (\Omega + t')) = 0.
\]
By translation invariance of the Haar measure, this is equivalent to
\[
\mathfrak{m}(\Omega \cap (\Omega + (t'-t))) = 0.
\]

Define the convolution \[f(x) = 1_{\Omega} * 1_{-\Omega}(x) = \int_{\Qp^d} 1_{\Omega}(y) 1_{\Omega}(y + x) d
\mathfrak{m}(y) = \mathfrak{m}(\Omega \cap (\Omega + x)).\]
Note that \(f(0) = \mathfrak{m}(\Omega) > 0\), and \(f\) is continuous on \(\Qp^d\).
Thus, there exists an integer \(n_0 \in \mathbb{Z}\) such that \(f(x) > 0\) for all \(x \in B(0,\ p^{n_0})\).

For distinct \(t, t' \in T\), \(\mathfrak{m}(\Omega \cap (\Omega + (t'-t))) = 0\)
implies \(f(t' - t) = 0\), so \(t' - t \notin B(0,\ p^{n_0})\). Thus
\[
|{t - t'}|_p > p^{n_0} \implies \inf_{t \neq t' \in T} |{t - t'}|_p \geq p^{n_0} > 0.
\]
Hence, \(T\) is uniformly discrete.
\end{proof}


If \(T\) is uniformly discrete, then for any compact set \(K\),
\[
\#(T \cap K) < \infty,
\]
so \(\mu_T\) is a Radon measure. The \textbf{bounded density} of \(T\) is defined, when it exists, as
\[
D(T) = \lim_{k \to \infty} \frac{\#(T \cap B(x_0, p^k))}{\mathfrak{m}(B(x_0, p^k))},
\]
where \(x_0\) is any base point. Due to translation invariance of the Haar measure, if the limit exists for one \(x_0\),
it exists for all and is independent of \(x_0\).

\begin{lemma}\label{lem:density}
Let \(\Omega \subset \mathbb{Q}_p^d\) be a bounded Borel set with positive Haar measure \(\mathfrak{m}(\Omega)>0\).
If \((\Omega, T)\) is a tiling pair, then
\[
D(T) = \frac{1}{\mathfrak{m}(\Omega)}.
\]
\end{lemma}
\begin{proof}
Let $x_0 \in \Qp^d$ be a base point. For $n \in \N$, define $K_n = x_0 + p^{-n}\Zp^d$. These sets satisfy
$K_n \subset K_{n+1}$, $\bigcup_{n \in \N} K_n = \Qp^d$ and  $\mathfrak{m}({K_n}) = p^{nd}$.
From the tiling definition,
\[
\sum_{t \in T} 1_\Omega(x-t) = 1, \quad \mathfrak{m}{\text{-a.e.}}
\]
Integrate both sides over $K_n$, by Tonelli's theorem for non-negative sums,
\[
\int_{K_n} \sum_{t \in T} 1_\Omega(x - t) dx = \int_{K_n} 1 dx = \mathfrak{m}({K_n}) = p^{nd}.
\]
Interchange sum and integral,
\begin{equation}\label{Tonelli}
\sum_{t \in T} \int_{K_n} 1_\Omega(x - t) dx = p^{nd}.
\end{equation}

By translation invariance of Haar measure, substitute $y = x - t$,
\[
\int_{K_n} 1_\Omega(x-t) dx = \int_{K_n + t} 1_\Omega(y) dy = \mathfrak{m}({(\Omega + t) \cap K_n}).
\]
Substitute back into (\ref{Tonelli}),
\begin{equation}\label{sum}
\sum_{t \in T} \mathfrak{m}((\Omega + t) \cap K_n) = p^{nd}.
\end{equation}
Let $N_n = \#\{t \in T : (\Omega + t) \cap K_n \neq \emptyset\}$. By tiling disjointness, the sets
$\{(\Omega + t) \cap K_n\}_{t \in T}$ are disjoint a.e., and
$\bigcup_{t\in T}((\Omega + t) \cap K_n) =K_n$.

Since $\Omega$ is bounded, $\Omega \subset B(0, p^M)$ for some $M \in \N$. For $n \geq M$,
\[
\mathfrak{m}((\Omega + t) \cap K_n)=\mathfrak{m}(\Omega + t) =
\mathfrak{m}(\Omega) \quad (\forall t \in T \text{ with } (\Omega+t) \cap K_n \neq \emptyset).
\]
Thus (\ref{sum}) simplifies to
\[
N_n \cdot \mathfrak{m}(\Omega) = p^{nd} \implies N_n = \frac{p^{nd}}{\mathfrak{m}(\Omega)}.
\]

For $n \geq M$,
If $t \in (T \cap K_n)$, then $\Omega + t \subset K_n$,
so \[t \in \{t \in T : (\Omega + t) \cap K_n \neq \emptyset\}.\]
If \(t \in \{t \in T : (\Omega + t) \cap K_n \neq \emptyset\},\) then $t = x - y$ for
$x \in K_n$ and $y \in \Omega$, so $\abs{t} \leq \max(p^n, p^M) = p^n$, hence $t \in T \cap K_n$.
Thus $N_n = \#(T \cap K_n)$ for $n \geq M$.
Substitute $N_n = \#(T \cap K_n)$ into above, we have
\[
\#(T \cap K_n) = \frac{p^{nd}}{\mathfrak{m}(\Omega)}.
\]
The density is
\[
D(T) = \lim_{n \to \infty} \frac{\#(T \cap K_n)}{\mathfrak{m}(K_n)} = \lim_{n \to \infty}
\frac{p^{nd}/\mathfrak{m}(\Omega)}{p^{nd}} = \frac{1}{\mathfrak{m}(\Omega)}.
\]

By translation invariance of Haar measure, replacing $x_0$ with any $x_1 \in \Qp^d$ gives a translate
$K_n' = x_1 + p^{-n}\Zp^d$. The count $\#(T \cap K_n')=\#((T - (x_1 - x_0)) \cap K_n)$,
and since $T - (x_1 - x_0)$ is also a tiling set for $\Omega$, its density is identical.
Thus the density is independent of $x_0$.
\end{proof}

%
%

\subsection{Finite harmonic analysis and spectral pairs in finite abelian groups}\label{subsec:finite-harmonic}


Let \(G\) be a finite abelian group. Its dual group \(\widehat{G}\) consists of all homomorphisms \(\psi : G \to \mathbb{C}^\times\).
For \(G = (\mathbb{Z}/p^n\mathbb{Z})^d\), the characters are given by
\[
\psi_\lambda(c) = e^{\frac{2\pi i}{p^n} \lambda \cdot c}, \quad \lambda, c \in G,
\]
where \(\lambda \cdot c = \sum_{j=1}^d \lambda_j c_j \bmod p^n\).

The space \(\ell^2(G)\) of complex-valued functions on \(G\) has inner product
\[
\langle f, g \rangle = \frac{1}{|G|} \sum_{c \in G} f(c) \overline{g(c)}.
\]
The characters \(\{\psi_\lambda\}_{\lambda \in G}\) form an orthonormal basis of \(\ell^2(G)\).

For a subset \(C \subset G\), we denote by \(\mathbf{1}_C\) its indicator function and by \(\delta_C\)
the uniform probability measure on \(C\)
\[
\delta_C = \frac{1}{|C|} \sum_{c \in C} \delta_c.
\]
Its Fourier transform on the finite group is
\[
\widehat{\delta_C}(\lambda) = \frac{1}{|C|} \sum_{c \in C} \overline{\psi_\lambda(c)} = \frac{1}{|C|} \sum_{c \in C} e^{-\frac{2\pi i}{p^n} \lambda \cdot c}.
\]


A subset \(C \subset G\) is called \textbf{spectral} if there exists \(\Lambda \subset G\) such that the characters \(\{\psi_\lambda|_C : \lambda \in \Lambda\}\) form an orthogonal basis of \(\ell^2(C)\). Equivalently, the matrix
\[
\left( \frac{1}{\sqrt{|C|}} e^{\frac{2\pi i}{p^n} \lambda \cdot c} \right)_{\lambda \in \Lambda, c \in C}
\]
is unitary. In this case, \((C, \Lambda)\) is called a \textbf{spectral pair} in \(G\).

The Fuglede's conjecture for finite abelian groups asks whether every spectral set in \(G\) also tiles \(G\) by translation, and vice versa. This conjecture has been confirmed for certain groups, including \(\mathbb{Z}/p^n\mathbb{Z}\) \cite{FFS} and \((\mathbb{Z}/p\mathbb{Z})^2\)~\cite{Iosevich-Mayeli-Pakianathan2015}, and is known to hold for all finite
abelian groups of order \(p^n q^m\) for $p<q$ and $m\le9$ or $n\le6$; $p^{m-2}<q^{4}$ \cite{M21,MK17}.


A key insight is that spectral pairs in the finite group \(G = (\mathbb{Z}/p^n\mathbb{Z})^d\) correspond to spectral pairs for certain atomic measures in \(\mathbb{Q}_p^d\). Specifically, if we embed \(G\) into \(\mathbb{Q}_p^d\) as a set of representatives in \(\{0,1,\dots,p^n-1\}^d\), then for \(\lambda, c \in G\), we have
\[
\psi_\lambda(c) = e^{\frac{2\pi i}{p^n} \lambda \cdot c} = e^{2\pi i \{\frac{\lambda}{p^n} \cdot c\}} = \chi\left( \frac{\lambda}{p^n} \cdot c \right) = \chi_{\lambda/p^n}(c).
\]
Thus, the characters of the finite group correspond to restrictions of \(p\)-adic characters to the embedded set \(C\).

\subsection{Technical lemmas}\label{subsec:tech-lemmas}

We introduce here a few more technical results that will be used in the proofs.

\begin{lemma}[Partition of unity]\label{lem:partition}
Let \(\Omega = \bigsqcup_{c \in C} (c + p^n\mathbb{Z}_p^d)\). Then for any \(x \in \mathbb{Q}_p^d\),
\[
\sum_{c \in C} \mathbf{1}_{c + p^n\mathbb{Z}_p^d}(x) = \mathbf{1}_\Omega(x).
\]
\end{lemma}

\begin{proof}
By definition, $\Omega$ is the disjoint union of the cosets $c + p^n \mathbb{Z}_p^d$ over $c \in C$.
If $x \in \Omega$, then there exists a unique $c_0 \in C$ such that $x \in c_0 + p^n \mathbb{Z}_p^d$. Hence,
    \[
    \mathbf{1}_{c_0 + p^n \mathbb{Z}_p^d}(x) = 1,
    \]
and for all $c \neq c_0$, $\mathbf{1}_{c + p^n \mathbb{Z}_p^d}(x) = 0$. Therefore, the sum equals $1$.
If $x \notin \Omega$, then $x$ does not belong to any of these cosets, so every term in the sum is $0$.
Thus, in both cases, the sum equals $\mathbf{1}_{\Omega}(x)$.
\end{proof}

\begin{lemma}[Fourier transform of a ball]\label{lem:fourier-ball}
For \(a \in \mathbb{Q}_p^d\) and \(n \in \mathbb{Z}\),
\[
\widehat{\mathbf{1}_{a + p^n\mathbb{Z}_p^d}}(\xi) = p^{-nd} \chi(-a \cdot \xi) \mathbf{1}_{p^{-n}\mathbb{Z}_p^d}(\xi).
\]
\end{lemma}

\begin{proof}
Recall the Fourier transform on $\mathbb{Q}_p^d$
\[
\widehat{f}(\xi) = \int_{\mathbb{Q}_p^d} f(x) \, \overline{\chi_\xi(x)} \, dx = \int_{\mathbb{Q}_p^d} f(x) \, \chi(-\xi \cdot x) \, dx.
\]
For $f = \mathbf{1}_{a + p^n \mathbb{Z}_p^d}$, we have
\begin{align*}
\widehat{\mathbf{1}_{a + p^n \mathbb{Z}_p^d}}(\xi)
&= \int_{a + p^n \mathbb{Z}_p^d} \chi(-\xi \cdot x) \, dx\\
&= \int_{p^n \mathbb{Z}_p^d} \chi(-\xi \cdot (a + y)) \, dy\\
&= \chi(-\xi \cdot a) \int_{p^n \mathbb{Z}_p^d} \chi(-\xi \cdot y) \, dy.
\end{align*}

Now, note that the remaining integral is the Fourier transform of $\mathbf{1}_{p^n \mathbb{Z}_p^d}$ at $\xi$.
For the compact open subgroup $H = p^n \mathbb{Z}_p^d$, we have
\[
\int_{H} \chi(-\xi \cdot y) \, dy = \mathfrak{m}(H) \cdot \mathbf{1}_{H^\perp}(\xi),
\]
where $$H^\perp = \{\xi \in \mathbb{Q}_p^d : \chi(-\xi \cdot y) = 1 \text{ for all } y \in H\}$$ is the annihilator of $H$.

Since $\chi(-\xi \cdot y) = e^{-2\pi i \{\xi \cdot y\}}$, we have $\chi(-\xi \cdot y) = 1$ for all $y \in H$ if and only if $\xi \cdot y \in \Zp^d$ for all $y \in p^n \mathbb{Z}_p^d$, i.e., $\xi \in p^{-n} \mathbb{Z}_p^d$. Thus, $H^\perp = p^{-n} \mathbb{Z}_p^d$.

Moreover, $\mathfrak{m}(p^n \mathbb{Z}_p^d) = p^{-nd}$ (since $\mathfrak{m}(\mathbb{Z}_p^d) = 1$). Therefore,
\[
\int_{p^n \mathbb{Z}_p^d} \chi(-\xi \cdot y) \, dy = p^{-nd} \, \mathbf{1}_{p^{-n} \mathbb{Z}_p^d}(\xi).
\]
Substituting back, we get
\[
\widehat{\mathbf{1}_{a + p^n \mathbb{Z}_p^d}}(\xi) = \chi(-\xi \cdot a) \cdot p^{-nd} \cdot \mathbf{1}_{p^{-n} \mathbb{Z}_p^d}(\xi).
\]
\end{proof}


\begin{lemma}[Orthogonality of characters on cosets]\label{lem:orthogonality-cosets}
Let \(H = p^n\mathbb{Z}_p^d\). For \(\xi, \eta \in \mathbb{Q}_p^d\),
\[
\int_{a + H} \chi_\xi(x) \overline{\chi_\eta(x)} \, dx = p^{-nd} \chi((\xi - \eta) \cdot a)
\mathbf{1}_{p^{-n}\mathbb{Z}_p^d}(\xi - \eta).
\]
\end{lemma}

\begin{proof}
This follows directly from Lemma \ref{lem:fourier-ball} and the properties of characters.
\end{proof}


\section{Proof of the main rsults}\label{mainresult}

In this section,
we provide a detailed  proof of Theorem~\ref{main}, breaking it down into clear steps and addressing all technical nuances.

\subsection{Tiling Equivalence}\label{subsec:tiling-proof}

We begin by restating and proving the tiling equivalence.

\begin{lemma}[Tiling equivalence]\label{lem:tiling-equiv}
Let \(\gamma \ge 0\) be an integer and \(C \subset (\mathbb{Z}/p^\gamma\mathbb{Z})^d\). Define the compact open set
\[
\Omega = \bigsqcup_{c \in C} \left( c + p^\gamma \mathbb{Z}_p^d \right),
\]
where \(C\) is identified with a set of representatives in \(\{0,1,\dots,p^\gamma-1\}^d \subset \mathbb{Z}_p^d\). Then \(\Omega\) tiles \(\Zp^d\) by translation if and only if \(C\) tiles \((\mathbb{Z}/p^\gamma\mathbb{Z})^d\) by translation.
\end{lemma}

\begin{proof}

Assume that \(C\) tiles the finite group \(G = (\mathbb{Z}/p^\gamma\mathbb{Z})^d\) with a tiling set \(T^* \subset G\), so that
\[
G = C \oplus T^*.
\]
That means every element of \(G\) can be written uniquely as \(c + t^*\) with \(c \in C\), \(t^* \in T^*\).

Let \(T \subset \mathbb{Z}_p^d\) be a set of representatives for \(T^*\) under the canonical projection \(\pi : \mathbb{Z}_p^d \to G\). A natural choice is to take \(T = T^*\) by viewing \(T^*\) as a subset of \(\{0,1,\dots,p^\gamma-1\}^d\).

\textbf{Covering.} Take any \(y \in \Z_p^d\). Since \(\pi(y) \in G\), there exist unique \(c \in C\) and \(t^* \in T^*\) such that \(\pi(y) = c + t^*\). Let \(t \in T\) be the representative of \(t^*\). Then \(\pi(y-t) = c\), so \(y - t \in c + p^\gamma\mathbb{Z}_p^d\). Hence,
\[
y =t + (y - t) \in  t + c + p^\gamma\mathbb{Z}_p^d  \subset \Omega + t.
\]
Thus, we have \(y \in \Omega + T\). So, \(\Zp^d \subset \Omega + T\).
The reverse inclusion is obvious, so \(\Zp^d = \Omega + \Lambda\).

\textbf{Disjointness.} Suppose that for two distinct \(t_1, t_2 \in T\),
we have \((\Omega + t_1) \cap (\Omega + t_2) \neq \emptyset\).
Then there exist \(c_1, c_2 \in C\) and \(z_1, z_2 \in p^\gamma\mathbb{Z}_p^d\) such that
\[
c_1 + t_1 + z_1 = c_2 + t_2 + z_2.
\]
Rearranging,
\[
c_1 + t_1 - c_2 - t_2 = z_2 - z_1 \in p^\gamma\mathbb{Z}_p^d.
\]
Applying the projection \(\pi\),
\[
\pi(c_1) + \pi(t_1) - \pi(c_2) - \pi(t_2) = 0 \quad \text{in } G.
\]
Thus, \(\pi(c_1) + \pi(t_1) = \pi(c_2) + \pi(t_2)\). By the uniqueness of representation in \(C \oplus T^*\),
we must have \(\pi(c_1) = \pi(c_2)\) and \(\pi(t_1) = \pi(t_2)\). Since \(C\) and \(T\) are sets of representatives,
this implies \(c_1 = c_2\) and \(t_1 = t_2\). This is contradicting the assumption \(t_1 \neq t_2\).
Therefore, the translates are pairwise disjoint, and we conclude that \(\Omega \oplus T = \Zp^d\).


Now assume that \(\Omega\) tiles \(\Zp^d\) by translation, i.e., there exists a set \(T \subset \Zp^d\)
such that
\[
\Zp^d = \Omega \oplus T.
\]

Now, define \(T^* = \pi(T)\), where \(\pi : \mathbb{Z}_p^d \to G\) is the projection. We claim that \(G = C \oplus T^*\).

\textbf{Covering in \(G\).} Let \(x \in G\). Choose a lift \(\tilde{x} \in \mathbb{Z}_p^d\) such that \(\pi(\tilde{x}) = x\). Since \(\mathbb{Z}_p^d = \Omega \oplus T\), there exist unique \(c \in C\), \(z \in p^\gamma\mathbb{Z}_p^d\), and \(t \in T\) such that
\[
\tilde{x} = c + z + t.
\]
Applying \(\pi\),
\[
x = \pi(\tilde{x}) = \pi(c) + \pi(t) = c + \pi(t) \in C + T^*.
\]
Hence, \(G \subset C + T^*\).

\textbf{Disjointness in \(G\).} Suppose \(c_1 + t_1^* = c_2 + t_2^*\) in \(G\) for some \(c_1, c_2 \in C\) and \(t_1^*, t_2^* \in T^*\). Choose representatives \(t_1, t_2 \in T\) such that \(\pi(t_1) = t_1^*\) and \(\pi(t_2) = t_2^*\). Then
\[
\pi(c_1 + t_1) = \pi(c_2 + t_2),
\]
so \(c_1 + t_1 - (c_2 + t_2) \in p^\gamma\mathbb{Z}_p^d\). Write
\[
c_1 + t_1 = c_2 + t_2 + z, \quad z \in p^\gamma\mathbb{Z}_p^d.
\]
Then
\[
c_1 + t_1 + p^\gamma\mathbb{Z}_p^d = c_2 + t_2 + p^\gamma\mathbb{Z}_p^d.
\]
Since \(\Omega \oplus T\) is a tiling, the sets \(\Omega + t_1\) and \(\Omega + t_2\) are
either disjoint or identical. The above equality shows they intersect, so they must be identical,
which implies \(t_1 = t_2\). Then from \(c_1 + t_1 = c_2 + t_1 + z\), we get
\(c_1 - c_2 = z \in p^\gamma\mathbb{Z}_p^d\), so \(\pi(c_1) = \pi(c_2)\),
hence \(c_1 = c_2\). Thus, the representation in \(C + T^*\) is unique.

Therefore, \(G = C \oplus T^*\), so \(C\) tiles \(G\) with tiling set \(T^*\).
\end{proof}
We observe that for each \(\alpha \in \mathbb{Q}_p^d\),
either \((\Omega + \alpha) \cap \mathbb{Z}_p^d = \emptyset\) or \((\Omega + \alpha) \subseteq\mathbb{Z}_p^d \).
Consequently, \(\Omega\) tiles \(\mathbb{Z}_p^d\)
if and only if it tiles \(\mathbb{Q}_p^d\). This yields the following corollary.

\begin{corollary}\label{corollary}
The set \(\Omega\) tiles \(\Qp^d\) by translation if and only if \(C\) tiles \((\mathbb{Z}/p^\gamma\mathbb{Z})^d\) by translation.
\end{corollary}

\subsection{Spectral Equivalence}\label{subsec:spectral-proof}

We now prove the spectral equivalence. The key step is the following lemma, which relates
spectrality in the finite group to spectrality in \(\mathbb{Q}_p^d\) for an associated atomic measure.

\begin{lemma}[Spectral equivalence for measures]\label{lem:spectral-equivalence}
Let \(G = (\mathbb{Z}/p^\gamma\mathbb{Z})^d\) with \(\gamma \ge 1\). Identify \(G\) with \(\{0,1,\dots,p^\gamma-1\}^d \subset \mathbb{Z}_p^d \subset \mathbb{Q}_p^d\). Let \(C, \Lambda \subset G\). Then \((C, \Lambda)\) is a spectral pair in \(G\)
if and only if \((\delta_C, \frac{1}{p^\gamma}\Lambda)\) is a spectral pair in \(\mathbb{Q}_p^d\), where \(\delta_C = \frac{1}{|C|}\sum_{c \in C} \delta_c\) is the uniform probability measure on \(C\), and \(\frac{1}{p^\gamma}\Lambda = \{\lambda/p^\gamma : \lambda \in \Lambda\} \subset \mathbb{Q}_p^d\).
\end{lemma}

\begin{proof}

Assume that \((C, \Lambda)\) is a spectral pair in \(G\). Then the matrix
\[
U = \left( \frac{1}{\sqrt{|C|}} e^{\frac{2\pi i}{p^\gamma} \lambda \cdot c} \right)_{\lambda \in \Lambda, c \in C}
\]
is unitary. In particular, the characters \(\psi_\lambda\) are orthogonal in \(\ell^2(C)\),
\[
\langle \psi_\lambda, \psi_{\lambda'} \rangle_{\ell^2(C)} = \frac{1}{|C|} \sum_{c \in C} \psi_{\lambda - \lambda'}(c) = \delta_{\lambda, \lambda'}.
\]
Using the inner product definition in finite groups, this orthogonality condition is equivalent to
\begin{equation}\label{dirac}
\widehat{\delta_C}\left( \frac{\lambda - \lambda'}{p^\gamma} \right) = 0 \quad \text{for all distinct } \lambda, \lambda' \in \Lambda.
\end{equation}

Now, consider the system \(E\left( \frac{1}{p^\gamma}\Lambda \right) = \left\{ \chi_{\lambda/p^\gamma} : \lambda \in \Lambda \right\}\) in \(L^2(\delta_C)\). For \(\lambda, \lambda' \in \Lambda\),
\begin{align*}
\int_{\mathbb{Q}_p^d} \chi_{\lambda/p^\gamma}(x) \overline{\chi_{\lambda'/p^\gamma}(x)} \, d\delta_C(x)
&= \frac{1}{|C|} \sum_{c \in C} \chi\left( \frac{\lambda - \lambda'}{p^\gamma} \cdot c \right)\\
&= \widehat{\delta_C}\left( \frac{\lambda - \lambda'}{p^\gamma} \right).
\end{align*}
By (\ref{dirac}), this is zero for \(\lambda \neq \lambda'\), and for \(\lambda = \lambda'\), it is \(\widehat{\delta_C}(0) = 1\). Thus, \(E\left( \frac{1}{p^\gamma}\Lambda \right)\) is an orthonormal system in \(L^2(\delta_C)\).

To show completeness, we need to verify that for every \(f \in L^2(\delta_C)\),
\[
\sum_{\lambda \in \Lambda} \left| \langle f, \chi_{\lambda/p^\gamma} \rangle_{L^2(\delta_C)} \right|^2 = \|f\|_{L^2(\delta_C)}^2.
\]
Since \(\delta_C\) is supported on the finite set \(C\), it suffices to check this for functions of the form
\(f = \mathbf{1}_{\{c\}}\) for \(c \in C\). But because the matrix \(U\) is unitary, the rows are also orthonormal, which gives
\[
\sum_{\lambda \in \Lambda} \left| \frac{1}{\sqrt{|C|}} e^{\frac{2\pi i}{p^\gamma} \lambda \cdot c} \right|^2 = 1,
\]
and more generally, the Parseval identity holds for all functions on \(C\). This implies that \(E\left( \frac{1}{p^\gamma}\Lambda \right)\) is complete in \(L^2(\delta_C)\). Hence, \((\delta_C, \frac{1}{p^\gamma}\Lambda)\) is a spectral pair.


Conversely, assume that \((\delta_C, \frac{1}{p^\gamma}\Lambda)\) is a spectral pair in \(\mathbb{Q}_p^d\). Then \(E\left( \frac{1}{p^\gamma}\Lambda \right)\) is an orthonormal basis for \(L^2(\delta_C)\). In particular, for distinct \(\lambda, \lambda' \in \Lambda\),
\[
\widehat{\delta_C}\left( \frac{\lambda - \lambda'}{p^\gamma} \right) = 0,
\]
which is equivalent to the orthogonality of \(\psi_\lambda\) and \(\psi_{\lambda'}\) in \(\ell^2(C)\). Moreover, since the system is complete, the number of elements in \(\Lambda\) must equal the dimension of \(L^2(\delta_C)\), which is \(|C|\). Thus, \(|\Lambda| = |C|\), and the orthogonal system \(\{\psi_\lambda\}_{\lambda \in \Lambda}\) is a basis for \(\ell^2(C)\). Therefore, \((C, \Lambda)\) is a spectral pair in \(G\).
\end{proof}


\begin{lemma}[Spectral equivalence for sets]\label{lem:spectral-equiv}
Let \(n \ge 0\) be an integer and \(C \subset (\mathbb{Z}/p^n\mathbb{Z})^d\).
Let
\[
\Omega = \bigsqcup_{c \in C} \left( c + p^n \mathbb{Z}_p^d \right),
\]
where \(C\) is identified with a set of representatives in \(\{0,1,\dots, p^n-1\}^d \subset \mathbb{Z}_p^d\).
Then \(\Omega\) is spectral in \(\Qp^d\) if and only if \(C\) is spectral in  \((\mathbb{Z}/p^n\mathbb{Z})^d\).
\end{lemma}
\begin{proof}
By Lemma \ref{lem:spectral-equivalence}, it suffices to prove the equivalence:
\begin{center}
\(\Omega\) is a spectral set in \(\Qp^d \Leftrightarrow \delta_C\) is a spectral measure in  \(\Qp^d\).
\end{center}
Recall that
\[\Omega = \bigsqcup_{c \in C} (c + p^n \mathbb{Z}_p^d),\]
and
\[
\mathbb{Q}_p^d =p^{-n}\mathbb{Z}_p^d\oplus p^{-n}\mathbb{L}^d=B(0, p^n)\oplus \mathbb{L}_n^d.
\]

Suppose \(\Omega\) is a spectral set in \(\Qp^d\) with spectrum \(\Gamma\subset\Qp^d\). Then
\(\Gamma=\Lambda_C+\mathbb{L}_n^d\), where  \(\Lambda_C\subset B(0, p^n)\) is some finite set.
By Lemma~\ref{lem:spectral-criterion}, we have
\begin{equation}\label{criterion}
\sum_{\gamma \in \Lambda_C+\mathbb{L}_n^d} |\widehat{\mathbf{1}_\Omega}(\gamma - \xi)|^2 =
\mathfrak{m}(\Omega)^2 \quad \forall \xi \in \mathbb{Q}_p^d.
\end{equation}
Using Lemma~\ref{lem:fourier-ball} for the Fourier transform of a ball and the disjoint union structure, we have
\begin{align}\label{equation}
\widehat{\mathbf{1}_\Omega}(\xi)\nonumber
&= \sum_{c \in C} \widehat{\mathbf{1}_{c + p^n \mathbb{Z}_p^d}}(\xi)\\ \nonumber
&= p^{-nd} \mathbf{1}_{p^{-n}\mathbb{Z}_p^d}(\xi) \sum_{c \in C} \chi(-c \cdot \xi) \\
&= p^{-nd} |C| \mathbf{1}_{p^{-n}\mathbb{Z}_p^d}(\xi) \widehat{\delta_C}(\xi).
\end{align}

Substituting (\ref{equation}) into (\ref{criterion}) and using \(\mathfrak{m}(\Omega) = |C| p^{-nd}\) gives
\begin{equation}\label{measurecriterion}
\sum_{\gamma \in \in \Lambda_C+\mathbb{L}_n^d} \mathbf{1}_{p^{-n}\mathbb{Z}_p^d}(\gamma - \xi)
|\widehat{\delta_C}(\gamma - \xi)|^2=1, \quad \forall \xi \in \mathbb{Q}_p^d.
\end{equation}

 Every \( \xi \in \mathbb{Q}_p^d\) can be written uniquely as \(\xi =\ell_\xi+z\) with
 \(\ell_\xi \in \mathbb{L}_n^d\) and \(z\in p^{-n}\mathbb{Z}_p^d\).
 The factor  \(\mathbf{1}_{p^{-n}\mathbb{Z}_p^d}(\gamma - \xi)\) is nonzero only if
 \(\gamma\in \ell_\xi+p^{-n}\mathbb{Z}_p^d\).
Since \(\Gamma=\Lambda_C+\mathbb{L}_n^d\), we have \(\gamma=\lambda+\ell\) with
\(\lambda\in \Lambda_C\) and \(\ell\in\mathbb{L}_n^d\). Then
\(\gamma\in \ell_\xi+p^{-n}\mathbb{Z}_p^d\) if and only if \(\ell=\ell_\xi\).
Thus, \(\gamma - \xi=\lambda-\ell-(\ell_\xi-z)=\lambda-z\).
Consequently,

\[\sum_{\gamma \in \Lambda_C+\mathbb{L}_n^d} \mathbf{1}_{p^{-n}\mathbb{Z}_p^d}(\gamma - \xi)
|\widehat{\delta_C}(\gamma - \xi)|^2=\sum_{\lambda \in \Lambda_C}
|\widehat{\delta_C}(\lambda - z)|^2=1.\]

Now, we need to show that
\[\sum_{\lambda \in \Lambda_C}|\widehat{\delta_C}(\lambda-z)|^2
=\sum_{\lambda \in \Lambda_C}|\widehat{\delta_C}(\lambda-\xi)|^2, \quad \forall \xi \in \mathbb{Q}_p^d.\]

Define $S: \mathbb{Q}_p^d \to \mathbb{R}$ by
\[
S(\xi) = \sum_{\lambda \in \Lambda_C} |\widehat{\delta_C}(\lambda - \xi)|^2.
\]
Our goal is to show that $S(\xi) = 1$ for all $\xi \in \mathbb{Q}_p^d$, which will imply that
\((\delta_C, \Lambda_C)\) is a spectral pair.

From the above, we already have,
\[
S(z) = 1 \quad \text{for all } z \in p^{-n}\mathbb{Z}_p^d.
\]
In particular, $S(0) = 1$.

Write
\[
|\widehat{\delta_C}(\xi)|^2 = \frac{1}{|C|^2} \sum_{c, c' \in C} \chi(-(c - c') \cdot \xi).
\]
Then
\begin{align*}
S(\xi) &= \sum_{\lambda \in \Lambda_C} \frac{1}{|C|^2} \sum_{c, c' \in C} \chi(-(c - c') \cdot (\lambda - \xi)) \\
&= \frac{1}{|C|^2} \sum_{c, c' \in C} \left( \sum_{\lambda \in \Lambda_C} \chi(-(c - c') \cdot \lambda) \right) \chi((c - c') \cdot \xi).
\end{align*}
Set
\[
A_{c,c'} = \sum_{\lambda \in \Lambda_C} \chi(-(c - c') \cdot \lambda).
\]
Then
\[
S(\xi) = \frac{1}{|C|^2} \sum_{c, c' \in C} A_{c,c'} \chi((c - c') \cdot \xi).
\]

Restrict to \(z\in p^{-n}\mathbb{Z}_p^d\):
\[
S(z) = \frac{1}{|C|^2} \sum_{c, c' \in C} A_{c,c'} \chi((c - c') \cdot z).
\]


%
Since $S(z)=1$   on $p^{-n}\mathbb{Z}_p^d$, and the characters $\{\chi((c - c') \cdot \cdot)\}$ are linearly independent on $p^{-n}\mathbb{Z}_p^d$, we must have $A_{c,c'} = 0$ for all $c \neq c'$. Otherwise, nontrivial character contributions would make $S(z)$ non-constant.

More formally, If $c \neq c'$, then $(c - c') \notin p^n\mathbb{Z}_p^d$, so $\chi((c - c') \cdot \cdot)$ is a nontrivial character on $p^{-n}\mathbb{Z}_p^d$. By linear independence of distinct characters, the coefficient $A_{c,c'}$ must be zero.

With $A_{c,c'} = 0$ for $c \neq c'$, we have
\[
S(\xi) = \frac{1}{|C|^2} \sum_{c \in C} A_{c,c} \chi(0) = \frac{1}{|C|^2} \sum_{c \in C} A_{c,c}.
\]
But $A_{c,c} = \sum_{\lambda \in \Lambda_C} \chi(0) = |\Lambda_C|$, so
\[
S(\xi) = \frac{1}{|C|^2} \cdot |C| \cdot |\Lambda_C| = \frac{|\Lambda_C|}{|C|}.
\]

Since $S(0) = 1$, we have $\frac{|\Lambda_C|}{|C|} = 1$, hence $|\Lambda_C| = |C|$. Therefore,
\[
S(\xi) = 1 \quad \text{for all } \xi \in \mathbb{Q}_p^d.
\]

Thus \((\delta_C, \Lambda_C)\) is a spectral pair,
by Lemma~\ref{lem:spectral-equivalence}, \(C\) is a spectral set in \(G\).

Conversely, assume \((\delta_C, \Lambda_C)\) is a spectral pair. That means
\begin{equation}\label{discrete}
\sum_{\gamma \in \Lambda_C} |\widehat{\delta_C}(\gamma - \xi)|^2=1, \quad \forall \xi \in \mathbb{Q}_p^d.
\end{equation}

Let \(\Gamma = \Lambda_C + \mathbb{L}_n^d.\)
We claim that \(\Gamma\) is a spectrum for \(\Omega\). 
Using (\ref{equation}) and (\ref{discrete})

\begin{align}\label{measurecriterion}
\sum_{\gamma \in \Lambda_C+\mathbb{L}_n^d} |\widehat{\mathbf{1}_\Omega}(\gamma - \xi)|^2 \nonumber
&=|C|^2 p^{-2nd}\sum_{\gamma \in \in \Lambda_C+\mathbb{L}_n^d} \mathbf{1}_{p^{-n}\mathbb{Z}_p^d}(\gamma - \xi)
|\widehat{\delta_C}(\gamma - \xi)|^2\\ \nonumber
&= |C|^2 p^{-2nd}\sum_{\gamma \in \Lambda_C} |\widehat{\delta_C}(\gamma - z)|^2\\ \nonumber
&=|C|^2 p^{-2nd}\\ \nonumber
&=\mathfrak{m}(\Omega)^2 \quad \forall \xi \in \mathbb{Q}_p^d,
\end{align}
where we used the same decomposition \(\xi=\ell_\xi+z\) as before.
Hence  \((\Omega, \Gamma)\) is a spectral pair.
\end{proof}

Combining Lemma~\ref{lem:tiling-equiv}, Corollary \ref{corollary} and Lemma~\ref{lem:spectral-equiv},
we have established Theorem~\ref{main}.

\textbf{Declarations}

\textbf{Conflict of interest:} The author have no relevant financial or non-financial interests to disclose.

\textbf{Data availability statement:} No data sets were generated or analysed during the current study.

\textbf{Funding:} This research is supported by NSF of China (Grant No. 12361015), and
by NSF of Xinjiang Uygur Autonomous Region, P. R. China (Grant No. 2025D01A09).


\begin{thebibliography}{00}


\bibitem{Fan}
A. H. Fan, {\em Spectral measures on local fields}, Springer Proc. Math. Stat., 150. Springer, Cham, 2015  15--35.

\bibitem{FFS}
A. H. Fan, S. L. Fan and R. X. Shi, {\em Compact open spectral sets in $\mathbb{Q}_p$,}
J. Funct. Anal. 271 (2016) 3628--3661.

\bibitem{FFLS}
A. H. Fan, S. L. Fan L. M. Liao and R. X. Shi, {\em Fuglede's conjecture  holds in $\mathbb{Q}_p$,}
Math. Ann. 375 (1-2) (2019)  315--341.

\bibitem{Fuglede1974}
B. Fuglede, {\em Commuting self-adjoint partial differential operators and
a group theoretic problem}, J. Funct. Anal., 16 (1974)  101--121.

\bibitem{Fuglede2001}
B. Fuglede, {\em Orthogonal exponentials on the ball}, Expo. Math. 19 (3) (2001) 267--272.

\bibitem{FMM2006}
B. Farkas, M. Matolcsi and P. M\'{o}ra, {\em On Fuglede's conjecture and the existence
of universal spectra,} J. Fourier Anal. Appl. 12 (5) (2006) 483--494.


\bibitem{Iosevich-Mayeli-Pakianathan2015}
A. Iosevich, A. Mayeli and J. Pakianathan,{\em The Fuglede conjecture holds in $\mathbb{Z}_p\times \mathbb{Z}_p$,}
Anal. PDE, 10 (4) (2017) 757--764.

\bibitem{Iosevich-Katz-Tao}A. Iosevich, N. Katz, T. Tao, {\em Convex bodies with a point of curvature do not have Fourier bases},
Amer. J. Math. 123(1) (2001) 115--120.

\bibitem{Iosevich-Katz-Tao2008}
A. Iosevich, N. Katz and T. Tao, {\em The Fuglede spectral conjecture holds for
convex planar domains,} Math. Res. Lett. 10 (5-6) (2003)  559--569.

\bibitem{Iosevich-Pedersen1998}
A. Iosevich and S. Pedersen, {\em Spectral and tiling properties of the unit cube,}
Int. Math. Res. Notices 16 (1998)  819--828.

\bibitem{Kolountzakis2000}
M. Kolountzakis, {\em Non-symmetric convex domains have no basis of exponentials,} Illinois J. Math. 44 (2000)  542--550.


\bibitem{KM2006}
M. Kolountzakis and M. Matolcsi, {\em Tiles with no spectra,} Forum Math., 18 (2006)  519--528.

\bibitem{KM2006a}
M. Kolountzakis and M. Matolcsi, {\em Complex Hadamard matrices and the
spectral set conjecture,} Collec. Math. Vol. Extra (2006)  281--291.

\bibitem{Laba2001}
I. Laba, {\em Fuglede's conjecture for a union of two intervals,} Proc. Amer. Math.
Soc. 129 (2001)  2965--2972.


\bibitem{Lagarias-Reed-Wang2000}
J.  Lagarias, J.  Reed and Y. Wang, {\em Orthonormal bases of exponentials
for the $n$-cube,} Duke Math. J. 103 (2000)  25--37.

\bibitem{LM21}
N.~Lev and M.~Matolcsi, {\em The Fuglede conjecture for convex domains is true in all dimensions},
Acta Math., 228(2) (2022) 385--420.

\bibitem{Matolcsi2005}
M. Matolcsi, {\em Fuglede's conjecture fails in dimension $4$,} Proc. Amer. Math. Soc.,
133 (2005)  3021--3026.

\bibitem{M21}
R. Malikiosis, {\em On the structure of spectral and tiling subsets of cyclic groups},
Forum Math., Sigma, 10 e23 (2022)  1--42.

\bibitem{MK17}
R. Malikiosis, M. Kolountzakis, {\em Fuglede's conjecture on cyclic groups of order {$p^n q$}},
Discrete Anal., Paper No. 12 (2017) 16.

\bibitem{Taibleson1975}
M. Taibleson, {\em Fourier Analysis on Local Fields,} Princeton University Press 2015.

\bibitem{Tao2004}
T. Tao, {\em Fuglede's conjecture is false in 5 and higher dimensions,} Math. Res. Lett., 11 (2004)  251--258.

\bibitem{VVZ1994}
V. Vladimirov, I. Volovich and E. Zelenov, {\em p-adic Analysis and Mathematical Physics,} World Scientific, 1994.

\end{thebibliography}
\end{document}